\documentclass{amsart}
\usepackage{amsmath}
\usepackage{mathrsfs}
\usepackage{amsfonts}
\usepackage{amssymb,color}
\usepackage{graphicx}
\usepackage[square,numbers]{natbib}
\usepackage{verbatim}
\usepackage{enumerate}
\newtheorem{theorem}{Theorem}

\newtheorem{lemma}[theorem]{Lemma}
\newtheorem{proposition}[theorem]{Proposition}

\theoremstyle{definition}
\newtheorem{definition}[theorem]{Definition}

\theoremstyle{remark}
\newtheorem{remark}[theorem]{Remark}

\numberwithin{theorem}{section}
\numberwithin{equation}{section}

\def\Xint#1{\mathchoice
   {\XXint\displaystyle\textstyle{#1}}%
   {\XXint\textstyle\scriptstyle{#1}}%
   {\XXint\scriptstyle\scriptscriptstyle{#1}}%
   {\XXint\scriptscriptstyle\scriptscriptstyle{#1}}%
   \!\int}
\def\XXint#1#2#3{{\setbox0=\hbox{$#1{#2#3}{\int}$}
     \vcenter{\hbox{$#2#3$}}\kern-.5\wd0}}

\def\dashint{\Xint-}

\newcommand{\dd}{\; \mathrm{d}}

\newcommand{\bbH}{\mathbb{H}}
\newcommand{\bbR}{\mathbb{R}}
\newcommand{\bbN}{\mathbb{N}}

\DeclareMathOperator*{\aplim}{aplim}

\begin{document}
\title[Lusin Approximation in the Heisenberg Group]{A $C^m$ Lusin Approximation Theorem for Horizontal Curves in the Heisenberg Group}
\author[Marco Capolli]{Marco Capolli}
\address[Marco Capolli]{Department of Mathematics, University of Trento, Via Sommarive 14, 38123 Povo (Trento), Italy}
\email[Marco Capolli]{marco.capolli@unitn.it}

\author[Andrea Pinamonti]{Andrea Pinamonti}
\address[Andrea Pinamonti]{Department of Mathematics, University of Trento, Via Sommarive 14, 38123 Povo (Trento), Italy}
\email[Andrea Pinamonti]{Andrea.Pinamonti@unitn.it}

\author[Gareth Speight]{Gareth Speight}
\address[Gareth Speight]{Department of Mathematical Sciences, University of Cincinnati, 2815 Commons Way, Cincinnati, OH 45221, United States}
\email[Gareth Speight]{Gareth.Speight@uc.edu}

\begin{abstract}
We prove a $C^m$ Lusin approximation theorem for horizontal curves in the Heisenberg group. This states that every absolutely continuous horizontal curve whose horizontal velocity is $m-1$ times $L^1$ differentiable almost everywhere coincides with a $C^m$ horizontal curve except on a set of small measure. Conversely, we show that the result no longer holds if $L^1$ differentiability is replaced by approximate differentiability. This shows our result is optimal and highlights differences between the Heisenberg and Euclidean settings.
\end{abstract}

\maketitle


\section{Introduction}
In mathematical analysis it is often useful to understand when a rough map can be approximated by a smoother one. For instance, Lusin's theorem asserts that every measurable function on $\bbR^n$ is continuous after removing a set of small measure from the domain. Another useful result states that every absolutely continuous curve in $\mathbb{R}^{n}$ has the $1$-Lusin property, which means that it coincides with a $C^{1}$ curve except for a set of small measure (Theorem \ref{classicallusinC1}). The position and velocity of absolutely continuous curves are related according to the Fundamental Theorem of Calculus, so these curves are important in analysis and geometry. Concerning higher regularity, if a curve in $\bbR^n$ is approximately differentiable of order $m$ almost everywhere (Definition \ref{approxdiff}), then it has the $m$-Lusin property allowing approximation by $C^m$ curves (Theorem \ref{classicallusinCm}) \cite{LT94}. In the present article we study the $m$-Lusin property for horizontal curves in the Heisenberg group, a non-Euclidean space with much geometric structure. The key difference between the Heisenberg group and Euclidean space is that in the Heisenberg group both the initial and the approximating curve must be horizontal, which means they are constrained to move in a smaller but still rich, family of directions.

In recent years, it has become clear that a large part of geometric analysis, geometric measure theory and real analysis in Euclidean spaces may be generalized to more general settings, see for example \cite{BLU07, Che99, CDPT07, FSS01, FSS03, HKST15, LPS17,Mon02, MPS17, Pan89, PS16, PSsurvey,PS18}. Carnot groups are Lie groups whose Lie algebra admits a stratification. This stratification gives rise to dilations and implies that points can be connected by absolutely continuous curves with tangents in a distinguished subbundle of the tangent bundle. These are the so called horizontal curves. Considering lengths of horizontal curves gives rise to the Carnot-Carath\'{e}odory distance and endows every Carnot group with a metric space structure. Moreover, every Carnot group has a natural Haar measure which respects the group translations and dilations. This plethora of structure makes the study of analysis and geometry in Carnot groups highly interesting \cite{BLU07, CDPT07, Mon02}. However, results in the Carnot setting can be very different to Euclidean ones since all such results must respect the horizontal structure of the Carnot group. The Heisenberg group is the simplest non-Euclidean Carnot group and admits an explicit representation in $\bbR^{2n+1}$ (Definition \ref{Heisenberg}) with $2n$ horizontal directions and one vertical direction.

In the Heisenberg group, the $1$-Lusin property is known to be true for all absolutely continuous horizontal curves with the requirement that the approximating curve can be chosen both $C^1$ and horizontal. More precisely, every absolutely continuous horizontal curve can be approximated by a $C^1$ horizontal curve (Theorem \ref{heisenbergc1Lusin}) \cite{S16}. A similar result holds in step two Carnot groups \cite{LS16} and more general pliable Carnot groups \cite{JS16, SS18}. However the natural analogue is not true in the Engel group, a Carnot group of step three \cite{S16}. This highlights that the approximation depends on the space considered and Euclidean results do not always extend to the Carnot group setting. Proving that smooth approximations exist is closely connected to validity of a Whitney extension theorem. In Euclidean spaces, the Whitney extension theorem (Theorem \ref{classicalWhitney}) \cite{Bie80, Whi34} characterizes when a collection of continuous functions defined on a compact set can be extended to a $C^m$ function on some larger set. To prove a Lusin approximation result from a Whitney extension theorem, one typically restricts to a large compact set where the original mapping satisfies the hypotheses of the Whitney extension theorem and then obtains a $C^m$ mapping which agrees with the starting map on a large set. To apply this idea in the Heisenberg group it is important to have an analogue of the Whitney extension theorem for curves in the Heisenberg group. Such a theorem is indeed known for $C^1$ horizontal curves in the Heisenberg group \cite{Zim18} and in more general spaces \cite{JS16, SS18}. Very recently it was also understood for $C^m$ horizontal curves in the Heisenberg group \cite{PSZ19}.

In the present paper we focus on the $m$-Lusin property in the Heisenberg group (Definition \ref{Lusinm}), investigating which horizontal curves can be approximated by $C^m$ horizontal curves. We now describe our main results.

Our first main result is Theorem \ref{thm1}. This asserts that if $\Gamma=(f,g,h)$ is an absolutely continuous horizontal curve in $\bbH^1$ with $f', g'$ almost everywhere $m-1$ times $L^1$ differentiable (Definition \ref{Lpdiff}), then $\Gamma$ has the $m$-Lusin property. This should be compared with the previously known analogue in Euclidean space (Theorem \ref{classicallusinCm}), which has the weaker hypothesis that $f,g,h$ are $m$ times approximately differentiable almost everywhere (Definition \ref{approxdiff}). Our arguments adapt the proof of Theorem \ref{classicallusinCm} from \cite{LT94} to the Heisenberg group using our stronger hypothesis to restrict to a compact set on which we can apply the $C^m$ Whitney extension theorem for horizontal curves in the Heisenberg group (Theorem \ref{CmWhitney}) recently proved \cite{PSZ19}.

Our second main result is Theorem \ref{thm2}, which illustrates the difference between the Heisenberg setting and the Euclidean setting. It also justifies the hypotheses of Theorem \ref{thm1}. In Theorem \ref{thm2} we construct an absolutely continuous horizontal curve $\Gamma$ in $\bbH^1$ such that $f,g,h$ are almost everywhere twice $L^p$ differentiable for all $p\geq 1$, $f',g',h'$ are almost everywhere once approximately differentiable, yet $\Gamma$ does not have the $2$-Lusin approximation property. This shows that the Euclidean hypothesis of twice approximate differentiability is not sufficient in $\bbH^1$ and that one really needs to assume differentiability properties of the derivatives $f',g',h'$ rather than only on $f,g,h$. Our argument is an explicit construction of a horizontal curve.

In the main results of this paper we restrict our attention to the first Heisenberg group $\bbH=\bbH^{1}$. We expect the natural analogue of Theorem \ref{thm1} is also true in higher dimensional Heisenberg groups $\bbH^{n}$ with similar proofs but more cumbersome notation. 

We now describe the structure of the paper. In Section \ref{Preliminaries} we recall key definitions involving the Heisenberg group, approximate derivatives, $L^p$ derivatives, $m$-Lusin property and Whitney extension theorems. In Section \ref{Facts} we prove preliminary results describing how $L^1$ differentiability behaves under integration or lifting to a horizontal curve and how approximate differentiability almost everywhere can be used to obtain a Whitney field. In Section \ref{Sectionthm1} we prove our first main result (Theorem \ref{thm1}). Finally in Section \ref{Sectionthm2} we prove our second main result (Theorem \ref{thm2}).

\vspace{.3cm}

\textbf{Acknowledgements:} Part of this work was done while A. Pinamonti and M. Capolli were visiting the University of Cincinnati (supported by funding from the Taft Research Center and the University of Trento) and while G. Speight was visiting the University of Trento (supported by funding from the University of Trento). A. Pinamonti and M. Capolli are members {\em Gruppo Nazionale per l'Analisi Ma\-te\-ma\-ti\-ca, la Probabilit\`a e le loro Applicazioni} (GNAMPA) of the {\em Istituto Nazionale di Alta Matematica} (INdAM). This work was also supported by a grant from the Simons Foundation (\#576219, G. Speight). The authors thank the referee for his/her valuable comments on the manuscript.

\section{Preliminaries}\label{Preliminaries}

\subsection{Heisenberg Group and Horizontal Curves}

\begin{definition}\label{Heisenberg}
The \emph{Heisenberg group} $\mathbb{H}^{n}$ is the Lie group represented in coordinates by $\mathbb{R}^{2n+1}$, whose points we denote by $(x,y,t)$ with $x,y\in \mathbb{R}^{n}$ and $t\in \mathbb{R}$. The group law is given by:
\[(x,y,t) (x',y',t')=\left(x+x',y+y',t+t'+2\sum_{i=1}^{n}(y_{i}x_{i}'-x_{i}y_{i}')\right).\]
\end{definition}

We equip $\mathbb{H}^{n}$ with left invariant vector fields
\[X_{i}=\partial_{x_{i}}+2y_{i}\partial_{t}, \quad Y_{i} = \partial_{y_{i}}-2x_{i}\partial_{t}, \quad 1\leq i\leq n, \quad T=\partial_{t}.\]
Here $\partial_{x_{i}}, \partial_{y_{i}}$ and $\partial_{t}$ denote the coordinate vectors in $\mathbb{R}^{2n+1}$, which may be interpreted as operators on differentiable functions. If $[\cdot, \cdot]$ denotes the Lie bracket of vector fields, then $[X_{i}, Y_{i}]=-4T$. Thus $\mathbb{H}^{n}$ is a Carnot group with horizontal layer $\mathrm{Span}\{X_{i}, Y_{i} : 1\leq i\leq n\}$ and second layer $\mathrm{Span}\{T\}$. In this paper we will mostly restrict ourselves to the first Heisenberg group $\bbH^1$ which we also denote by $\bbH$.

\begin{definition}
A vector in $\mathbb{R}^{2n+1}$ is \emph{horizontal} at $p \in \mathbb{R}^{2n+1}$ if it is a linear combination of the vectors $X_{i}(p), Y_{i}(p), 1\leq i\leq n$.

An absolutely continuous curve $\gamma$ in the Heisenberg group is \emph{horizontal} if, at almost every point $t$, the derivative $\gamma'(t)$ is horizontal at $\gamma(t)$.
\end{definition}

\begin{lemma}\label{lift}
An absolutely continuous curve $\gamma\colon [a,b]\to \mathbb{R}^{2n+1}$ is a horizontal curve in the Heisenberg group if and only if, for $t\in [a,b]$:
\[\gamma_{2n+1}(t)=\gamma_{2n+1}(a)+2\sum_{i=1}^{n}\int_{a}^{t} (\gamma_{i}'\gamma_{n+i}-\gamma_{n+i}'\gamma_{i}).\]
\end{lemma}

We will use Lemma \ref{lift} repeatedly throughout the paper. In the first Heisenberg group $\bbH=\bbH^1$, the relevant equations for an absolutely continuous curve to be horizontal simplify to
\begin{equation}\label{liftH}
\gamma_{3}(t)=\gamma_{3}(a)+2\int_{a}^{t} (\gamma_{1}'\gamma_{2}-\gamma_{2}'\gamma_{1}).
\end{equation}

Clearly Lemma \ref{lift} implies that for any horizontal curve $\gamma$ we have
\begin{equation}\label{liftH'}
\gamma_{2n+1}'(t) = 2\sum_{i=1}^{n} (\gamma_{i}'(t)\gamma_{n+i}(t)-\gamma_{n+i}'(t)\gamma_{i}(t)) 
\quad
\text{for a.e. } t \in [a,b].
\end{equation}
If we assume that $\gamma$ is $C^1$, this equality holds for every $t \in [a,b]$. If we further assume that $\gamma$ is $C^m$ for some $m > 1$, then, for $1 \leq k \leq m$, we may write 
\begin{equation}\label{HigherHoriz}
D^k\gamma_{2n+1}(t) = \sum_{i=1}^{n} \mathcal{P}^k\left(\gamma_{i}(t),\gamma_{n+i}(t),\gamma_{i}'(t),\gamma_{n+i}'(t),\dots,D^k\gamma_{i}(t),D^k\gamma_{n+i}(t)\right)
\end{equation}
for all $t\in [a,b]$ where $\mathcal{P}^k$ is a polynomial determined by the Leibniz rule. For a $C^m$ horizontal curve $\gamma$ in the first Heisenberg group $\bbH=\bbH^1$, the equations simplify to
\begin{equation}\label{HigherHoriz}
\gamma_{3}^{k}=2\sum_{i=0}^{k-1}  {{k-1}\choose{i}} (\gamma_{1}^{k-i}\gamma_{2}^{i}-\gamma_{2}^{k-i}\gamma_{1}^{i}) \qquad \mbox{for }1\leq k\leq m.
\end{equation}

A classical result about the $1$-Lusin property in Euclidean spaces can be found for example in \cite{AT04, EG91}.

\begin{theorem}\label{classicallusinC1}
Let $\gamma\colon [a,b]\to \bbR^n$ be absolutely continuous. Then for every $\varepsilon>0$ there exists a $C^1$ map $\Gamma\colon [a,b]\to \bbR^{n}$ such that
\[\mathcal{L}^1(\{ x\in [a,b]: \Gamma(t)\neq \gamma(t)\} )<\varepsilon.\]
\end{theorem}

%
%
%

We also recall the the analogous result for the $1$-Lusin property in Heisenberg groups \cite{S16}.

\begin{theorem}\label{heisenbergc1Lusin}
Absolutely continuous horizontal curves in $\bbH^n$ have the $1$-Lusin property.
\end{theorem}

It is also known that absolutely continuous horizontal curves have the $1$-Lusin property in step two Carnot groups \cite{LS16}, in pliable Carnot groups \cite{JS16} and in suitable sub-Riemannian manifolds \cite{SS18}.

\subsection{Approximate Differentiability and Integral Differentiability}

Recall that if $f\colon \bbR^{d}\to \bbR$, $x\in \bbR^{d}$ and $l\in \bbR$, then $\aplim_{y\to x}f(x) = l$ means that for every $\varepsilon>0$ the set
\[\{ y\in \bbR^d : |f(y)-l|\leq \epsilon \}\]
has density one at $x$, i.e.
\[ \lim_{R\to 0} \frac{\mathcal{L}^{d}(B(x,R)\cap \{ y\in \bbR^d : |f(y)-l|\leq \epsilon \})}{\mathcal{L}^{d}(B(x,R))}=1.\]

\begin{definition}\label{approxdiff}
Given $x\in\bbR^d$ and $k\in \bbN$, we say that a function $u\colon \bbR^{d}\to \bbR$ is \emph{$m$ times approximately differentiable at $x$} if there exists a polynomial $P^m_{u,x}$ of degree at most $m$ such that
\begin{equation}\label{eq.approxdiff}
    \aplim_{y\to x}\frac{|u(y)-P^m_{u,x}(y)|}{|y-x|^m} = 0.
\end{equation}
\end{definition}

\begin{remark}
The polynomial $P^m_{u,x}$ in Definition \ref{approxdiff} is uniquely determined and can be expressed in the form
\begin{equation}\label{eq.approxpolyn}
P^m_{u,x}(y) = \sum_{|\alpha|\leq m} \frac{u_{\alpha}(x)}{|\alpha|!}(y-x)^{\alpha}
\end{equation}
for some $u_{\alpha}(x)\in \mathbb{R}$ \cite{LT94}.
\end{remark}

As a special case of a recent result from \cite{LT94} we get a $C^m$ version of the Lusin property.

\begin{theorem}\label{classicallusinCm}
Suppose $\gamma \colon [a,b]\to \bbR^n$ is measurable and approximately differentiable of order $m$ almost everywhere. Then for every $\varepsilon>0$ there exists a $C^m$ map $\Gamma\colon [a,b]\to \bbR^{n}$ such that
\[\mathcal{L}^1(\{ x\in [a,b]: \Gamma(t)\neq \gamma(t)\} )<\varepsilon.\]
\end{theorem}

In the Heisenberg group we give the following definition of Lusin property for horizontal curves.

\begin{definition}\label{Lusinm}
An absolutely continuous horizontal curve $\Gamma\colon [a,b]\to \bbH^n$ is said to have the \emph{Lusin property of order $m$} if for every $\epsilon>0$ there exists a $C^m$ horizontal curve $\widetilde{\Gamma}\colon [a,b]\to \bbH^{n}$ such that
\[\mathcal{L}^{1}(\{x\in [a,b]: \widetilde{\Gamma}(x)\neq \Gamma(x)\})<\varepsilon.\]
\end{definition}

We will also refer to the Lusin property of order $m$ as the $m$-Lusin property. 

%
%

Throughout this paper we use the usual notation for integral averages
\[\dashint_{A} f=\frac{1}{\mathcal{L}^{d}(A)}\int_{A}f\]
for any $A\subset \bbR^{d}$ and $f\colon A\to \bbR$ for which the expression is well defined.

\begin{definition}\label{Lpdiff}
Let $u\colon \bbR^{d} \to \bbR$, $x\in\bbR^d$, $p\in[1,\infty)$, and $m\in \mathbb{N}$. 

We say that $u$ is \emph{$m$ times $L^p$ differentiable at $x$} if there exists a polynomial $P^m_{u,x}$ on $\bbR^d$ of degree at most $m$ such that
\begin{equation}\label{eq.lpderivative}
    \left[ \dashint_{B(x,\rho)} |u(y) - P^m_{u,x}(y)|^p \dd y \right]^{1/p} = o(\rho^m).
\end{equation}
\end{definition}


%


\begin{remark}
As noted for instance in \cite{ABC14}, if $u$ is $m$ times $L^p$ differentiable at $x$ then $u$ is also $m$ times approximately differentiable at $x$ with the same choice of $P^m_{u,x}$.
\end{remark}

\subsection{Jets and Whitney Extension}

\begin{definition}\label{jet}
A \emph{jet} of order $m\in \bbN$ on a set $K\subset \mathbb{R}$ consists of a collection of $(m+1)-$continuous functions $F=(F^{k})_{k=0}^{m}$ on $K$.

Given such a jet $F$ and $a\in K$, the \emph{Taylor polynomial of order $m$ of $F$ at $a$} is
\[T_{a}^{m}F(x)=\sum_{k=0}^{m} \frac{F^{k}(a)}{k!}(x-a)^{k}
\quad
\text{for all } x \in \mathbb{R}.\]
If $m$ or $a$ are clear from the context, we may write $TF$ for the Taylor polynomial. We will also use the notation $F(x)$ for $F^{0}(x)$.
\end{definition}

Given a jet $F$ of order $m$ on $K\subset \mathbb{R}$, for $a\in K$ and $0\leq k\leq m$ we define
\[(R_{a}^{m}F)^{k}(x)=F^{k}(x)-\sum_{\ell=0}^{m-k}\frac{F^{k+\ell}(a)}{\ell!}(x-a)^{\ell}
\quad
\text{for all } x \in \mathbb{R}.\]

\begin{definition}\label{whitneyfield}
A jet $F$ of order $m$ on $K$ is a \emph{Whitney field of class $C^m$ on $K$} if, for every $0\leq k\leq m$, we have
\[(R_{a}^{m}F)^{k}(b)=o(|a-b|^{m-k})\]
as $|a-b|\to 0$ with $a,b \in K$.
\end{definition}

We now recall the classical Whitney extension theorem in the special case that the domain is a subset of $\mathbb{R}$ \cite{Whi34}.

\begin{theorem}[Classical Whitney extension theorem]\label{classicalWhitney}
Let $K$ be a closed subset of an open set $U\subset \mathbb{R}$. Then there is a continuous linear mapping $W$ from the space of Whitney fields of class $C^m$ on $K$ to $C^m(U)$ such that
\[D^{k}(WF)(x)=F^{k}(x) \quad \mbox{for $0\leq k\leq m$ and $x\in K$},\]
and $WF$ is $C^{\infty}$ on $U\setminus K$.
\end{theorem}

We now recall the Whitney extension theorem for $C^m$ horizontal curves in $\bbH$ from \cite{PSZ19}. Suppose $F, G, H$ are jets of order $m$ on $K\subset \bbR$. For $a,b\in K$, we define the \emph{area discrepancy}
\begin{align}\label{Aab}
 A(a,b)&:=H(b)-H(a)-2\int_{a}^{b}((T_{a}^mF)'(T_{a}^mG)-(T_{a}^mG)'(T_{a}^mF))\\
\nonumber &\qquad +2F(a)(G(b)-T_{a}^mG(b))-2G(a)(F(b)-T_{a}^mF(b))
\end{align}
and the \emph{velocity}
\begin{equation}\label{Vab}
V(a,b):= (b-a)^{2m} + (b-a)^{m} \int_{a}^{b} \left(|(T_{a}^mF)'|+|(T_{a}^mG)'| \right).
\end{equation}

We say that jets $(F,G,H)$ of order $m$ on $K$ extend to a $C^m$ horizontal curve $(f,g,h)\colon \mathbb{R}\to \mathbb{H}$ if $(f,g,h)\colon \mathbb{R}\to \mathbb{H}$ is a $C^m$ horizontal curve such that $f^{i}|_{K}=F^{i}$, $g^{i}|_{K}=G^{i}$ and $h^{i}|_{K}=H^{i}$ for $0\leq i\leq m$.

\begin{theorem}\label{CmWhitney}
Let $K\subset \bbR$ be compact and $F, G, H$ be jets of order $m$ on $K$. 
Then $(F,G,H)$ extends to a $C^m$ horizontal curve $(f,g,h)\colon \bbR \to \bbH$ if and only if
\begin{enumerate}
\item $F, G, H$ are Whitney fields of class $C^m$ on $K$,
\item For $1\leq k\leq m$ the following equation holds at all points of $K$
\begin{equation}\label{ODEeq}
H^{k}=2\sum_{i=0}^{k-1}  {{k-1}\choose{i}} (F^{k-i}G^{i}-G^{k-i}F^{i}),
\end{equation}
\item $A(a,b)/V(a,b)\to 0$ uniformly as $(b-a) \to 0$ with $a,b\in K$.
\end{enumerate}
\end{theorem}

Finally we state for future use the following fact about polynomials from \cite{PSZ19}.

\begin{lemma}\label{intmax}
Let $P$ be a polynomial of degree $n$, $a<b$, and $\|P\|_{\infty}:=\max_{[a,b]}|P|$. Then
\[ \frac{1}{8n^{2}} \|P\|_{\infty}\leq \dashint_{a}^{b} |P|\leq \|P\|_{\infty}.\]
\end{lemma}

\section{Facts about Approximate Derivatives and $L^{1}$ Derivatives}\label{Facts}

In this section we prove several lemmas which will be useful later in the paper.

\begin{lemma}\label{intL1}
Let $f\colon [a,b]\to \mathbb{R}$ be absolutely continuous and $m\geq 2$.

Suppose $f'$ is $m-1$ times $L^{1}$ differentiable at a point $x\in (a,b)$ with $L^1$ derivative given by the polynomial $P_{f,x}^{m-1}$ of degree at most $m-1$. Then $f$ is $m$ times $L^{1}$ differentiable at $x$ with $L^1$ derivative $Q_{f,x}^{m}$ of degree at most $m$ defined by $Q_{f,x}^m(y):=f(x)+\int_{x}^{y}P_{f,x}^{m-1}(t) \dd t$.
\end{lemma}

\begin{proof}
Denote $P=P_{f,x}^{m-1}$ and define $Q=Q_{f,x}^m$ by $Q_{f,x}^m(y):=f(x) + \int_x^y P(t)\dd t$. Let $\varepsilon>0$. From the definition of $L^1$ differentiability we have for all sufficiently small $\rho>0$
\begin{equation*}
\dashint_{B(x,\rho)} |f'(t) - P(t)|\dd t \leq \varepsilon \rho^{m-1}/2.
\end{equation*}
Absolute continuity gives for all $y\in B(x,\rho)$,
\begin{equation*}
\begin{split}
|f(y) - Q(y)| &= \left| f(x) + \int_x^y f'(t)\dd t - \left( f(x) + \int_x^y P(t)\dd t \right) \right| \\
& = \left| \int_x^y (f'(t) - P(t))\dd t \right| \\
&\leq \int_{B(x,\rho)} |f'(t) - P(t)|\dd t \\
& \leq \varepsilon \rho^m.
\end{split}
\end{equation*}
Hence given $\varepsilon>0$, we have for all sufficiently small $0<\rho<1$
\begin{equation*}
\begin{split}
\dashint_{B(x,\rho)} |f(y)-Q(y)|\dd y \leq \dashint_{B(x,\rho)} \varepsilon \rho^m = \varepsilon \rho^m.
\end{split}
\end{equation*}
This proves the lemma.
\end{proof}

\begin{lemma}\label{vertL1}
Suppose $(f,g,h)\colon [a,b]\to \mathbb{H}$ is an absolutely continuous horizontal curve in $\bbH$ and $f', g'$ are $m-1$ times $L^{1}$ differentiable at a point $x\in [a,b]$ for some $m\geq 2$. Then $h$ is $m$ times $L^{1}$ differentiable at $x$. More precisely, denote
\[R(y):=h(x)+2\int_{x}^{y}(P'Q-Q'P),\]
where $P, Q$ are the $L^{1}$ derivatives of order $m$ of $f, g$ respectively which exist by Lemma \ref{intL1}. Let $\widetilde{R}$ be the polynomial of degree at most $m$ such that $R(y)-\widetilde{R}(y)$ is divisible by $(y-x)^{m+1}$. Then $\widetilde{R}$ is the $L^1$ derivative of $h$ of order $m$ at $x$.
\end{lemma}

\begin{proof}
Let $R$ be defined as in the statement of the lemma. Fix $0<\varepsilon<1$. Then there exists $\delta>0$ such that for all $0<\rho<\delta$ we have
\[\dashint_{B(x,\rho)}|f-P|\leq \varepsilon \rho^{m}, \qquad \dashint_{B(x,\rho)}|f'-P'|\leq \varepsilon \rho^{m-1},\]
and
\[\dashint_{B(x,\rho)}|g-Q|\leq \varepsilon \rho^{m}, \qquad \dashint_{B(x,\rho)}|g'-Q'|\leq \varepsilon \rho^{m-1}.\]
Let $0<\rho<\delta$ and $y\in B(x,\rho)$. We estimate as follows, using the fact $(f,g,h)$ is a horizontal curve and \eqref{liftH'},
\begin{align*}
h(y)-R(y)&= h(x)+ \int_{x}^{y}h' - h(x)-2\int_{x}^{y}(P'Q-Q'P)\\
&=2\int_{x}^{y}((f'g-P'Q)+(Q'P-g'f)).
\end{align*}
We estimate the first term as follows
\begin{align*}
\left| 2 \int_{x}^{y}(f'g-P'Q)\right| &\leq 2\int_{B(x,\rho)} |f'g-P'Q|\\
&= 4\rho \dashint_{B(x,\rho)} |f'g-P'Q|.
\end{align*}
Since $f'g-P'Q=(f'-P')g+P'(g-Q)$ and $g, P'$ are continuous hence bounded on $[a,b]$, we can continue our estimate as follows
\begin{align*}
4\rho \dashint_{B(x,\rho)} |f'g-P'Q|&\leq 4\rho \left( \|g\|_{\infty} \dashint_{B(x,\rho)}|f'-P'| + \|P'\|_{\infty}\dashint_{B(x,\rho)}|g-Q|\right)\\
&\leq 4\rho \left( \|g\|_{\infty}\varepsilon \rho^{m-1}+\|P'\|_{\infty}\varepsilon \rho^{m}\right)\\
&\leq C\varepsilon \rho^{m}
\end{align*}
for a constant $C$ independent of $y$ and $\rho$. The estimate of $2\int_{x}^{y}(Q'P-g'f)$ is similar. Hence we obtain $|h(y)-R(y)|\leq C\varepsilon \rho^{m}$ for all $0<\rho<\delta$. Consequently
\[\dashint_{B(x,\rho)}|h-R|\leq C\varepsilon \rho^{m}.\]
To conclude we notice that if $\widetilde{R}$ is the polynomial of degree at most $m$ defined in the statement of the lemma then for some constant $C$ independent of $\rho<1$ we have
\begin{align*}
\dashint_{B(x,\rho)}|h-\widetilde{R}| &\leq \dashint_{B(x,\rho)} |h-R|+\dashint_{B(x,\rho)}|\widetilde{R}-R|\\
&\leq \dashint_{B(x,\rho)}|h-R|+C\rho^{m+1}.
\end{align*}
Hence
\[\dashint_{B(x,\rho)}|h-\widetilde{R}|=o(\rho^{m})\]
so
$h$ is $m$ times $L^{1}$ differentiable at $x$ with derivative $\widetilde{R}$.
\end{proof}

We next prove Proposition \ref{approxWhitney} which shows that approximate differentiability almost everywhere leads to Whitney fields on large compact sets. Our argument is adapted from \cite{LT94} where a similar result is proved under slightly different assumptions. As in \cite{LT94} we need the following lemma by De Giorgi \cite{C64}.

\begin{lemma}[De Giorgi]
Let $E$ be a measurable subset of the ball $B(x,r)$ in $\mathbb{R}^n$ such that $\mathcal{L}^{n}(E)\geq Ar^n$ for some constant $A>0$. Then for each $m\in \bbN$ there is a positive constant $C$, depending only on n,m and $A$, such that for each polynomial $p$ of degree at most $m$ and for every multi-index $\alpha$
\[
|D^{\alpha}p(x)| \leq \frac{C}{r^{n+|\alpha|}}\int_E |p(y)|dy.
\]
\end{lemma}

\begin{proposition}\label{approxWhitney}
Let $u\colon [a,b]\to \mathbb{R}$ be measurable and $m$ times approximately differentiable almost everywhere. Let the approximate derivative at almost every point $x$ be denoted by
\[P_{u,x}^m(y) = \sum_{i=0}^m\frac{u_i(x)}{i!}(y-x)^i.\]
Then for every $\varepsilon>0$ there exists a compact set $K\subset [a,b]$ with $\mathcal{L}^1([a,b]\setminus K) \leq \varepsilon$ such that $\Gamma=(u_{i})_{i=0}^{m}$ is a $C^m$ Whitney field on $K$.
\end{proposition}

\begin{proof}
It is proven in \cite{LT94} that all the functions $u_{i}$ are measurable under the given hypotheses. Let $0<\delta < 1$ and $0<\varepsilon<1$ be fixed for the moment. For every $x\in[a,b]$ where $u$ is approximately differentiable and $r>0$, define
\[
W(x,r):=\{y \in[a,b]\cap [x-r,x+r]: |u(y)-P_{u,x}^m(y)|> \delta |x-y|^m \}.
\]
Each set $W(x,r)$ is measurable because all the $u_i$ are measurable. We can write $W(x,r) = \{ y\in[a,b]:(x,y)\in T(r) \}$, where
\[
T(r) := \{ (x,y)\in [a,b]\times [a,b] : |x-y|<r, |u(y) - P_{u,x}^m(y)|>\delta|x-y|^m \}.
\]
Since $T$ is measurable, it follows $x\mapsto \mathcal{L}^{1}(W(x,r))$ is a measurable function of $x$. For $n\in \bbN$ define the sets
\begin{equation}\label{eq.bdeltaset}
B_{n} := \{ x\in[a,b] : \mathcal{L}^{1}(W(x,r))\leq r/4 \quad \mbox{for all } r\leq 1/n \}.
\end{equation}
Since $x\mapsto \mathcal{L}^{1}(W(x,r))$ is a measurable function of $x$ and $\mathcal{L}^{1}(W(x,r))$ is monotonic in $r$ for each fixed $x$, it is easy to show that the sets $B_{n}$ are measurable. Clearly $B_{n}\subset B_{n+1}$ for every $n$. Since $u$ is $m$ times approximately differentiable almost everywhere, it follows $\mathcal{L}^{1}([a,b]\setminus \bigcup_{n=1}^{\infty} B_{n}) = 0$.
Consider two points $x,y\in B_{n}$ with $x\leq y$ and $|x-y|\leq 1/n$. Let $r = |y-x|$ and define the measurable sets
\[
S(x,y) := [x,y]\setminus (W(x,r)\cup W(y,r)).
\]
Then
\[
\mathcal{L}^{1}(S(x,y)) \geq |y-x|-\mathcal{L}^{1}(W(x,r))-\mathcal{L}^{1}(W(y,r))\geq r/2.
\]
Define the polynomial $q:= P_{u,y}^m - P_{u,x}^m$. For $z\in S(x,y)$ we estimate $|q(z)|$ as follows
\[
|q(z)| \leq |P_{u,y}^m(z) - u(z)| + |u(z) - P_{u,x}^m(z)| \leq \delta(|z-y|^m + |x-z|^m)\leq 2\delta r^m.
\]
We apply De Giorgi's Lemma to the polynomial $q$ with $E=S(x,y)$ and $A=1/2$ to obtain for every $k$
\[
|D^k q(y)| = |u_k(y) - D^k P_{u,x}^m(y)| \leq \frac{C}{r^{1+k}}\int_{S(x,y)} |q(z)|\dd z \leq 2C\delta r^{m-k}.
\]
Recall $\varepsilon>0$ was fixed earlier. Since $\mathcal{L}^{1}([a,b]\setminus \bigcup_{n=1}^{\infty} B_{n}) = 0$ and the sets $B_{n}$ are increasing, we may choose $N\in \bbN$ such that $\mathcal{L}^{1}([a,b]\setminus B_{N})\leq \varepsilon/2$. We then choose $K$ a compact subset of $B_{N}$ with $\mathcal{L}^{1}([a,b]\setminus K)\leq \varepsilon$. Now we recall the dependence of $K$ on $\varepsilon, \delta$ and denote $K=K(\varepsilon, \delta)$ and $N=N(\varepsilon, \delta)$. The set $K(\varepsilon, \delta)$ has the following two properties for a constant $C$ depending only on $m$:
\begin{enumerate}
\item $\mathcal{L}^{1}([a,b]\setminus K(\varepsilon, \delta))\leq \varepsilon$,
\item For every $0\leq k\leq m$ and $x,y\in K(\varepsilon, \delta)$ with $|x-y|\leq 1/N(\varepsilon, \delta)$ we have
\[|u_k(y) - D^k P_{u,x}^m(y)|\leq 4C \delta |x-y|^{m-k}.\]
\end{enumerate}
We now put our compact sets together. Fix $\varepsilon>0$ and define
\[K=\bigcap_{n=1}^{\infty} K(\varepsilon/2^{n},1/n).\]
Using (1) for the sets $K(\varepsilon/2^{n},1/n)$, we estimate the measure of $K$ as follows
\begin{align*}
\mathcal{L}^{1}([a,b]\setminus K)&\leq \sum_{n=1}^{\infty} \mathcal{L}^{1}([a,b]\setminus K(\varepsilon/2^{n},1/n))\\
&\leq \sum_{n=1}^{\infty}\varepsilon/2^{n}\\
&=\varepsilon.
\end{align*}
Using (2) for the sets $K(\varepsilon/2^{n},1/n)$, we see that $K$ has the following property. Whenever $0\leq k\leq m$ and $x,y\in K$ satisfy $|x-y|\leq N(\varepsilon/2^{n}, 1/n)$ for some $n\in \bbN$,
\[|u_k(y) - D^k P_{u,x}^m(y)|\leq 4C |x-y|^{m-k}/n.\]
In other words, for every $0\leq k\leq m$ we have
\[|u_k(y) - D^k P_{u,x}^m(y)| =o(|x-y|^{m-k})\]
as $|x-y|\to 0$ with $x,y\in K$. Hence $\Gamma=(u_{i})_{i=0}^{m}$ is a $C^m$ Whitney field on $K$.
\end{proof}

%

\section{$C^m$ Horizontal Lusin Approximation for Horizontal Curves with $L^1$ Differentiable Velocity}\label{Sectionthm1}

In this section we prove our first main theorem. Before giving the statement we first recall that if $f\colon [a,b]\to \mathbb{R}$ is $m$ times $L^{1}$ differentiable at a point $x\in [a,b]$, then we denote the $L^{1}$ derivative at $x$ by
\[P_{f,x}^m(y) = \sum_{i=0}^m \frac{f_{i}(x)}{i!}(y-x)^i,\]
where $f_{i}(x)\in \mathbb{R}$ for $0\leq i\leq m$. Also, if a function $f\colon [a,b]\to \mathbb{R}$ is absolutely continuous and $f'$ is $m-1$ times $L^{1}$ differentiable at a point $x\in [a,b]$, then $f$ is $m$ times $L^{1}$ differentiable at $x$ with derivative given by Lemma \ref{intL1}.

\begin{theorem}\label{thm1}
Let $I\subset\mathbb{R}$ be an interval and $\Gamma = (f,g,h):I\to\bbH$ be an absolutely continuous horizontal curve such that $f'$ and $g'$ are $m-1$ times $L^{1}$ differentiable at almost every point of $I$. Then $\Gamma$ has the $m$-Lusin property. Further, for every $\eta>0$ there is a $C^m$ horizontal curve $\widetilde{\Gamma}=(\widetilde{f},\widetilde{g},\widetilde{h})\colon I\to \bbH$ such that
\[\mathcal{L}^{1} \left(\bigcup_{k=0}^{m}\{x\in I: \widetilde{f}^{k}(x)\neq f_{k}(x) \, \mbox{ or } \, \widetilde{g}^{k}(x)\neq g_{k}(x) \, \mbox{ or } \,  \widetilde{h}^{k}(x)\neq h_{k}(x)\}\right)<\eta.\]
\end{theorem}

\begin{proof}
Using Lemma \ref{intL1} it follows that $f$ and $g$ are $m$ times $L^{1}$ differentiable almost everywhere. By Lemma \ref{vertL1} we also know that $h$ is $m$ times $L^{1}$ differentiable almost everywhere. At almost every $x\in I$ denote the $L^1$ derivative of $f$ by
\[P_{f,x}^m(y) = \sum_{k=0}^m\frac{f_k(x)}{k!}(y-x)^k\]
where the $f_k$ are measurable functions by \cite{LT94}. Similarly define the $L^{1}$ derivatives $P_{g,x}^m$ and $P_{h,x}^m$ with coefficients $g_k(x)$ and $h_k(x)$ at almost every point $x$, which are measurable functions of $x$.

Fix $\eta>0$. Choose a compact set $K\subset I$ satisfying $\mathcal{L}^{1}(I\setminus K)<\eta$ with the following properties:
\begin{enumerate}
\item the jets $F, G, H$ defined on $K$ by
\[F^k = f_k|_{K},\ G^k = g_k|_{K} \ \text{ and }\ H^k = h_k|_{K}  \ \text{ for } 0\leq k \leq m\]
are Whitney fields of class $C^m$ on $K$.
\item For every $\varepsilon>0$ there is $\delta>0$ such that if $a,b\in K$ with $|b-a|<\delta$ then 
\begin{equation}\label{uniformL1}
\dashint_a^b |f'-(T_{a}^{m}F)'| \leq \epsilon (b-a)^{m-1} \quad \mbox{and} \quad  \dashint_a^b |g'-(T_{a}^{m}G)'| \leq \epsilon (b-a)^{m-1}.
\end{equation}
\end{enumerate}
The first property above is possible using Proposition \ref{approxWhitney}. To obtain the second property we use the almost everywhere $(m-1)$ times $L^{1}$ differentiability of $f'$ and $g'$, Lemma \ref{intL1}, elementary measure theory, and the fact that $P_{f,a}^m=T_{a}^{m}F$ and $P_{g,a}^m=T_{a}^{m}G$. We now show that the hypotheses of Theorem \ref{CmWhitney} hold for the jets $F, G, H$ on the compact set $K$.



\smallskip

\emph{Verification of Theorem \ref{CmWhitney}(1).} This follows directly from the definition of $K$.

\smallskip

\emph{Verification of Theorem \ref{CmWhitney}(2).} We need to check \eqref{ODEeq}, which we recall states
\[H^{k}=2\sum_{i=0}^{k-1}  {{k-1}\choose{i}} (F^{k-i}G^{i}-G^{k-i}F^{i}) \quad \mbox{on }K\mbox{ for }1\leq k\leq m.\]
Fix $a\in K$ and let $TF= T^m_a F$, $TG = T^m_a G$, $TH = T^m_a H$ for simplicity. Using Lemma \ref{vertL1}, we know
\[(P^{m}_{h,a})'=2((P^{m}_{f,a})'(P^{m}_{g,a})-(P^{m}_{g,a})'(P^{m}_{f,a}))+S_{a}'(y),\]
where $S_{a}(y)$ is a polynomial divisible by $(y-a)^{m+1}$. Hence
\[(TH)'=2((TF)'(TG)-(TG)'(TF))+S_{a}'(y).\]
Differentiating the Taylor polynomials as was done to derive \eqref{ODEeq} yields
\[(TH)^{k}=2\sum_{i=0}^{k-1}  {{k-1}\choose{i}} ((TF)^{k-i}(TG)^{i}-(TG)^{k-i}(TF)^{i}) + S_{a}^{k} \quad \mbox{on }K\mbox{ for }1\leq k\leq m,\]
where the polynomial $S_{a}^{k}(y)$ is divisible by $(y-a)^{m+1-k}$. In particular, $S_{a}^{k}(a)=0$ for $1\leq k\leq m$. Since the Taylor polynomials are based at $a$, $(TF)^{i}(a)=F^{i}(a)$ for $0\leq i\leq m$ and similarly for $G$ and $H$. Hence substituting in $a$ we obtain
\[H^{k}(a)=2\sum_{i=0}^{k-1}  {{k-1}\choose{i}} (F^{k-i}(a)G^{i}(a)-G^{k-i}(a)F^{i}(a))\]
for all $a\in K$ and $1\leq k\leq m$ as required.

\smallskip

\emph{Verification of Theorem \ref{CmWhitney}(3).} Given $0<\varepsilon<1$ fixed, let $\delta>0$ be chosen as above. Fix $a,b \in K$ with $0<b-a<\delta$. For ease of notation we write $TF = T^m_a F$ and $TG = T^m_a G$. For simplicity we will consider only the case
\[F(a) = G(a) = H(a) = 0.\]
Otherwise one can use the translation invariance of $A(a,b)$ and $V(a,b)$ as in \cite{PSZ19}. Since $F(a) = G(a) = H(a) = 0$, $A(a,b)$ is of the form
\[
A(a,b) = H(b) - H(a) - 2\int_a^b ((TF)'TG - TF(TG)').
\]
Since $(f,g,h)$ is a horizontal curve, we have
\[
H(b) - H(a) =h(b)-h(a)= 2\int_a^b (f'g - fg').
\]
We estimate $|A(a,b)|$ as follows
\begin{align*}
&\left|H(b) - H(a) - 2\int_a^b ((TF)'TG - TF(TG)')\right| \\
&\qquad \leq 2\left( \int_a^b |f'g - (TF)'TG| + \int_a^b |fg' - TF(TG)'| \right).
\end{align*}
We will show how to estimate the first term after the inequality, the second one will follow by changing the roles of $f$ and $g$. First we pass to the average
\[
\int_a^b |f'g - (TF)'TG| = (b-a)\dashint_a^b |f'g - (TF)'TG|
\]
and then we decompose the argument as
\[
f'g - (TF)'TG = (f'-(TF)')(g-TG)+(f'-(TF)')TG + (g-TG)(TF)'.
\]
We then obtain
\[
    \begin{split}
        (b-a)\dashint_a^b |f'g - (TF)'TG| \leq (b-a)&\left[ \left( \dashint_a^b |f'-(TF)'| \right)||g-TG ||_{\infty} \right.\\
        &+ \left( \dashint_a^b |f'-(TF)'| \right)||TG||_{\infty}\\
        &+ \left. \left( \dashint_a^b |g-TG| \right) ||(TF)'||_{\infty} \right].
    \end{split}
\]
From \eqref{uniformL1} we obtain
\[
\left( \dashint_a^b |f'-(TF)'| \right) \leq \epsilon (b-a)^{m-1}.
\]
Absolute continuity of $g$, the Fundamental Theorem of Calculus, and \eqref{uniformL1} gives
\[
||g-TG||_{\infty} \leq (b-a) \left( \dashint_a^b |g'-(TG)'| \right) \leq \epsilon (b-a)^m.
\]
Using Lemma \ref{intmax} we have
\[
||(TF)'||_{\infty} \leq C \dashint_a^b |(TF)'|
\]
for some constant $C\geq 1$ depending only on $m$. Using $TG(a)=G(a)=g(a)=0$ and again the Fundamental Theorem of Calculus, we have
\[
||TG||_{\infty} \leq \int_a^b |(TG)'|.
\]
Combining all together we get
\begin{align*}
\int_a^b |f'g - (TF)'TG| &\leq \epsilon^2 (b-a)^{2m} + \epsilon (b-a)^m \int_a^b |(TG)'| + C \epsilon (b-a)^m \int_a^b |(TF)'|.\\
&\leq C\varepsilon V(a,b)
\end{align*}
By doing the same computation with $f$ and $g$ switched we obtain
\[|A(a,b)|\leq 4C\varepsilon V(a,b)\]
whenever $a,b\in K$ with $0<b-a<\delta$. This yields Theorem \ref{CmWhitney}(3).


\smallskip

\emph{Conclusion.} We have shown that the jets $F, G, H$ satisfy the hypotheses of Theorem \ref{CmWhitney} on the compact set $K$. Hence $\Gamma=(F,G,H)$ extends to a $C^m$ horizontal curve $\widetilde{\Gamma}=(\widetilde{f}, \widetilde{g}, \widetilde{h})\colon I\to \mathbb{H}^{1}$ satisfying
\[\widetilde{f}^{k}|_{K}=F^{k},\quad \widetilde{g}^{k}|_{K}=G^{k},\quad \widetilde{h}^{k}|_{K}=H^{k} \quad \mbox{for }0\leq k\leq m.\]
From the definition of the compact set $K$ and the jets $F, G, H$ we have
\begin{align*}
&\mathcal{L}^{1} \left(\bigcup_{k=0}^{m}\{x\in I: \widetilde{f}^{k}(x)\neq f_{k}(x) \, \mbox{ or } \, \widetilde{g}^{k}(x)\neq g_{k}(x) \, \mbox{ or } \,  \widetilde{h}^{k}(x)\neq h_{k}(x)\}\right)\\
&\quad \leq \mathcal{L}^{1}(I\setminus K)\\
&\quad<\eta.
\end{align*}
This completes the proof of the theorem.
\end{proof}

\section{A Horizontal Curve with no Lusin Approximation}\label{Sectionthm2}

In this section we prove our second main theorem, which justifies the hypotheses of Theorem \ref{thm1} and highlights the difference between the settings of Euclidean space and the Heisenberg group.

\begin{theorem}\label{thm2}
There exists $\Gamma=(f,g,h)\colon [0,1]\to \bbH$ which is absolutely continuous and horizontal with the following properties:
\begin{enumerate}
\item Almost everywhere the maps $f, g, h$ are twice $L^p$ differentiable for all $p\geq 1$,
\item Almost everywhere the maps $f', g', h'$ are once approximately differentiable,
\item $\Gamma$ does not admit a $C^2$ horizontal Lusin approximation.
\end{enumerate}
\end{theorem}

We use the remainder of this section to prove Theorem \ref{thm2}.

\subsection{Construction of the Horizontal Curve}

\subsubsection{Parameters for the Construction}
Fix decreasing sequences $h_{n}, \lambda_{n}>0$ with
\begin{equation}\label{hlambda1}
\quad \sum_{n=1}^{\infty} 2^{n}\lambda_{n}<\infty, \quad h_{n}/\lambda_{n}\to 0, \quad 4^{n}h_{n}\to \infty, \quad \frac{1}{\lambda_{n+1}^{2}}\sum_{k=n+1}^{\infty} 2^{k-n}h_{k}^{2}\to 0.
\end{equation}
One possible choice is $h_{n}=1/3^n$ and $\lambda_{n}=(2/5)^n$. A consequence of \eqref{hlambda1} is
\begin{equation}\label{hlambda2}
\sum_{n=1}^{\infty} 2^{n}h_{n}<\infty.
\end{equation}
We next fix a decreasing sequence $w_{n}>0$ such that
\begin{equation}\label{w1}
w_{n}\leq 1/2^{6n}, \qquad \frac{1}{\lambda_{n+1}}\sum_{k=n+1}^{\infty} 2^{k-n}w_{k}\to 0,
\end{equation}
and
\begin{equation}\label{w2}
 \frac{1}{\lambda_{n+1}^{2p+1}}\sum_{k=n+1}^{\infty}2^{k-n}w_{k}h_{k}^{p}\to 0 \qquad \mbox{for every }p\geq 1.
 \end{equation}
 This is possible since $w_{n}$ can be chosen very small compared to $h_{n}$ and $\lambda_{n}$.

%

\subsubsection{The Sets $I_{n}$ and $I$}
For each $n\geq 1$ we inductively define sets $I_{n}\subset [0,1]$, each a disjoint union of finitely many open intervals, as follows. Firstly, $I_{1}$ is the open interval with center $1/2$ and radius $w_{1}$. Once $I_{1}, I_{2}, \ldots, I_{n}$ are defined, we define $I_{n+1}$ as the union of those open intervals $J$ with the following properties:
\begin{itemize}
\item $J$ has center $k/2^{n+1}$ for some integer $k$ with $0<k<2^{n+1}$,
\item $J$ has radius $w_{n+1}$,
\item $J$ does not intersect $I_{1}\cup I_{2}\cup \cdots \cup I_{n}$.
\end{itemize}
Define $I=\cup_{n=1}^{\infty}I_{n}$. The set $I_{n}$ is a disjoint union of at most $2^{n-1}$ intervals of length $2w_{n}$. Hence, since $w_{n}\leq 1/2^{6n}$,
\begin{equation}\label{mI}
\mathcal{L}^1(I)\leq \sum_{n=1}^{\infty} \mathcal{L}^1(I_{n}) \leq \sum_{n=1}^{\infty} 2^{n}w_{n}\leq 1/31.
\end{equation}

\subsubsection{Definition of the Horizontal Components} 
We now define $f,g\colon [0,1]\to \mathbb{R}$ which will be the first two components of the curve. In $[0,1]\setminus I$ we set $f$ and $g$ to be identically $0$. Otherwise we proceed as follows. Suppose $J$ is one of the finitely many disjoint open intervals chosen in the definition of $I_{n}$ for some $n\geq 1$. Divide $J$ into $4$ adjacent disjoint equally sized intervals labelled from left to right
\[J_{1}=(p_{1}, p_{2}), \quad J_{2}=[p_{2}, p_{3}], \quad J_{3}=[p_{3}, p_{4}], \quad J_{4}=(p_{4},p_{5}).\]
The maps $f, g$ are piecewise linear functions in $J$ defined as follows:
\begin{enumerate}
\item In $J_{1}$, $f$ is identically $0$ and $g$ is linear with $g(p_1)=0$, $g(p_2)=h_n$.
\item In $J_{2}$, $f$ is linear with $f(p_2)=0$, $f(p_3)=h_n$ and $g$ is identically $h_{n}$.
\item In $J_{3}$, $f$ is identically $h_{n}$ and $g$ is linear with $g(p_3)=h_n$, $g(p_4)=0$.
\item In $J_{4}$, $f$ is linear with $f(p_4)=h_{n}$, $f(p_5)=0$ and $g$ is identically $0$.
\end{enumerate}

\subsubsection{Absolute Continuity of the Horizontal Components}
Clearly $f$ and $g$ are differentiable at all but finitely many points of $I_{n}$ for each $n$, hence at all but countably many points of $I$. Our first task is to prove that $f$ and $g$ are differentiable at almost every point of $[0,1]\setminus I$. Before doing so we prove a lemma which roughly states that at almost every point of $[0,1]\setminus I$ the maps $f$ and $g$ do not see `big jumps' unexpectedly close to $x$.

For $x\in \bbR$ and $S\subset \bbR$ we denote $d(x,S):=\inf\{|x-y| :y\in S\}$. For  $n\geq 1$, define
\[A_{n}=\{x\in [0,1]\setminus I:d(x,I_{1}\cup \cdots \cup I_{n})<\lambda_{n}\}\]
and let
\[A:=\limsup A_{n}\subset [0,1]\setminus I.\]
By definition of the limit superior, for any $x\in [0,1]\setminus (I\cup A)$, there exists $N(x)>0$ such that $n>N(x)$ implies
\[d(x,I_{1}\cup \cdots \cup I_{n})\geq \lambda_{n}.\]
Roughly speaking, this states that if $x\in [0,1]\setminus (I\cup A)$ then on small scales near to $x$ one sees only relatively small intervals. We will use this fact repeatedly later.

\begin{lemma}\label{limsup0}
The set $A$ has Lebesgue measure zero.
\end{lemma}

\begin{proof}
The set $I_{i}$ consists of $2^{i-1}$ intervals and is contained in $I$. Hence
\[\mathcal{L}^1(\{x\notin I\colon d(x,I_{i})<\lambda_{n}\} ) \leq 2^{i-1}2\lambda_{n}=2^{i}\lambda_{n}.\]
Hence
\[\mathcal{L}^1(A_{n})=(2+2^{2}+\cdots+2^{n})\lambda_{n}=2\lambda_{n}(2^n-1).\]
Since $\sum_{n=1}^{\infty}2^n\lambda_{n}<\infty$ it follows $\sum_{n=1}^{\infty} \mathcal{L}^1(A_{n})<\infty$. The Borel Cantelli lemma gives the conclusion.
\end{proof}

\begin{lemma}
For every $x\in (0,1)\setminus (I\cup A)$, $f$ and $g$ are differentiable at $x$ with $f'(x)=g'(x)=0$.
\end{lemma}

\begin{proof}
Fix a point $x$ as in the statement of the lemma and corresponding $N(x)>0$  such that $n>N(x)$ implies
\[d(x,I_{1}\cup \cdots \cup I_{n})\geq \lambda_{n}.\]
For all $t$ sufficiently small there is $n>N(x)$ such that $\lambda_{n+1}\leq |t|< \lambda_{n}$. Then
\[d(x, I_{1}\cup \cdots \cup I_{n})\geq \lambda_{n}> |t|.\]
This implies $x+t\notin I_{1}\cup \cdots \cup I_{n}$. By definition of $f$ we see $0\leq f(x+t)\leq h_{n+1}$. Since $x\notin I$ we have $f(x)=0$ and so
\[\left| \frac{f(x+t)-f(x)}{t}\right| \leq \frac{h_{n+1}}{\lambda_{n+1}}.\]
Since $h_{n}/\lambda_{n}\to 0$, it follows that $f$ is differentiable at $x$ with $f'(x)=0$. The argument is the same for $g$.
\end{proof}

We have now shown that $f$ and $g$ are differentiable almost everywhere on $[0,1]$.

\begin{proposition}
The maps $f,g\colon [0,1]\to \mathbb{R}$ are absolutely continuous.
\end{proposition}

\begin{proof}
Suppose $J$ is one of the intervals chosen in the construction of $I_{n}$ for some $n\geq 1$. Then for any $x\in J$ we have
\[|f'(x)|\leq h_{n}/(w_{n}/4)=4h_{n}/w_{n}.\]
Since $\mathcal{L}^1(J)\leq 2w_{n}$ it follows that $\int_{J}|f'|\leq 8h_{n}$. Since there are at most $2^{n-1}$ disjoint intervals in the construction of $I_{n}$, we have $\int_{I_{n}}|f'|=2^{n+2}h_{n}$ for every $n\geq 1$. Since $f'=0$ almost everywhere outside $I$, we deduce,
\[\int_{0}^{1}|f'|\leq \sum_{n=1}^{\infty} 2^{n+2}h_{n}<\infty.\]
Hence $f'$ is integrable on $[0,1]$.

We now claim
\begin{equation}\label{FTC}
f(b)-f(a)=\int_{a}^{b}f' \quad \mbox{whenever }a<b.
\end{equation}
Clearly \eqref{FTC} is satisfied if $a$ and $b$ belong to a common chosen interval $J$ from the definition of $I$. Indeed, $f$ is piecewise linear and hence absolutely continuous inside any such interval. Suppose this is not the case. By splitting the integral if necessary, to prove \eqref{FTC} we may assume $a, b\notin I$. If $J=[c,d]$ is any interval chosen in the construction of $I$ which is contained in $(a,b)$, then
\[\int_{J}f'=f(d)-f(c)=0.\]
There are countably many such intervals and $f'=0$ at almost every point outside $I$. Hence $\int_{a}^{b}f'=0$. Since $a,b\notin I$ we have $f(b)=f(a)=0$. Hence \eqref{FTC} holds. This proves that $f$ is absolutely continuous. The argument for $g$ is the same.
\end{proof}

\subsubsection{Vertical Component of the Curve}
Since $f,g$ are bounded and $f',g'$ are integrable, the products $f'g$ and $g'f$ are integrable. We define $h\colon [0,1]\to \bbR$ by
\[h(x):=2\int_{0}^{x}(f'g-g'f) \qquad \mbox{for }x\in [0,1].\]
Clearly $h$ is absolutely continuous on $[0,1]$. By Lemma \ref{lift}, $\Gamma:=(f,g,h)$ is an absolutely continuous horizontal curve. It is easy to check that $h$ is piecewise linear since each interval chosen in the construction of $I$. We also record the following fact for later.

\begin{lemma}\label{hinc}
Suppose $J=(a,b)$ is one of the connected components of $I_{n}$. Then
\[h(b)-h(a)= 4h_{n}^{2}.\]
\end{lemma}

\begin{proof}
Since $f(a)=f(b)=0$ we know
\[h(b)-h(a)=2\int_{a}^{b}(f'g-g'f)=4\int_{a}^{b}f'g.\]
From the construction of $f$ and $g$ and the fact $(b-a)/4=w_{n}/2$ we obtain
\[h(b)-h(a)=4(w_{n}/2)(h_{n}/(w_{n}/2))h_{n}=4h_{n}^{2}.\]
\end{proof}

\subsection{Differentiability of the Horizontal Curve}

\begin{proposition}\label{lpf}
At almost every point $x\in [0,1]$, the maps $f, g, h\colon [0,1]\to \bbR$ are twice $L^p$ differentiable at $x$ for all $p\geq 1$. For every point $x\in (0,1)\setminus (I\cup A)$, the second order $L^p$ derivatives of $f, g, h$ at $x$ are identically $f(x)=0$, $g(x)=0$, and $h(x)$ (possibly non-zero) respectively.
\end{proposition}

\begin{proof}
Recall that $f, g, h$ are piecewise linear inside each of the countably many intervals whose disjoint union is $I$. Hence $f, g, h$ are twice $L^p$ differentiable for all $p\geq 1$ at all but countably many points of $I$. Suppose $x\notin (I\cup A)$. To show $f$ is twice $L^p$ differentiable at $x$ we will show that for every $p\geq 1$
\[\lim_{t\to 0} \frac{1}{t^{2p+1}} \int_{[x-t,x+t]}|f(y)|^{p}\dd y=0.\]
Using the definition of $A$, we may choose $N(x)>0$  such that $n>N(x)$ implies
\[d(x,I_{1}\cup \cdots \cup I_{n})\geq \lambda_{n}.\]
Recall that $\sum_{n=1}^{\infty} 2^{n}\lambda_{n}<\infty$ which implies $\lambda_{n}\leq 1/2^{n}$ for all sufficiently large $n$. Given any $t>0$ sufficiently small, we may choose $n>N(x)$ with
\[\lambda_{n+1}\leq t<\lambda_{n}\leq 1/2^n.\]
This implies
\[[x-t,x+t]\cap (I_{1}\cup \cdots \cup I_{n})=\varnothing.\]
The interval $[x-t,x+t]$  has length at most $2\lambda_{n}\leq 1/2^{n-1}$. Since the intervals in $I_{k}$ have centers separated by at least distance $1/2^{k}$, it follows that $[x-t,x+t]$ can intersect at most $2^{k-n+1}$ intervals from $I_{k}$ for $k>n$. Recall that $x\notin I$ gives $f(x)=0$, $|f(y)|\leq h_{k}$ for $y$ in an interval from $I_{k}$, and that $t>\lambda_{n+1}$. We have
\begin{align*}
\frac{1}{t^{2p+1}} \int_{[x-t,x+t]}|f(y)|^{p}\dd y &= \frac{1}{t^{2p+1}} \int_{[x-t,x+t]\cap I}|f(y)|^{p}\dd y \\
&\leq \frac{2}{t^{2p+1}}\sum_{k=n+1}^{\infty} 2^{k-n+1}w_{k}h_{k}^{p}\\
&\leq \frac{2}{\lambda_{n+1}^{2p+1}}\sum_{k=n+1}^{\infty}2^{k-n+1}w_{k}h_{k}^{p}.
\end{align*}
The previous line converges to $0$ as $n\to \infty$ for every $p\geq 1$ by definition of the sequences $w_{k}, h_{k}, \lambda_{k}$. The argument for $g$ is exactly the same. Finally to show $h$ is twice $L^p$ differentiable at $x$ we will show that for every $p\geq 1$
\[\lim_{t\to 0} \frac{1}{t^{2p+1}} \int_{[x-t,x+t]}|h(y)-h(x)|^{p}\dd y=0.\]
Recall that $[x-t,x+t]$ can intersect at most $2^{k-n+1}$ intervals from $I_{k}$ for $k>n$. Hence, using Lemma \ref{hinc}, for any $y\in [x-t,x+t]$ we have
\[|h(y)-h(x)|\leq \sum_{k=n+1}^{\infty} 2^{k-n+3}h_{k}^{2}.\]
Hence
\begin{align*}
\frac{1}{t^{2p+1}} \int_{[x-t,x+t]}|h(y)-h(x)|^{p}\dd y &\leq \frac{2}{t^{2p}}\left( \sum_{k=n+1}^{\infty} 2^{k-n+3}h_{k}^{2} \right)^{p}\\
&\leq 2 \left( \frac{8}{\lambda_{n+1}^{2}} \sum_{k=n+1}^{\infty} 2^{k-n}h_{k}^{2} \right)^{p}.
\end{align*}
We conclude by noticing the last line converges to $0$ as $n\to \infty$ for every $p\geq 1$.
\end{proof}

\begin{proposition}
The maps $f', g', h'$ are once approximately differentiable almost everywhere. In particular, $f'$ and $g'$ have approximate derivative $0$ at every point of $(0,1)\setminus (I\cup A)$.
\end{proposition}

\begin{proof}
Approximate differentiability of $f', g', h'$ at all but countably many points of $I$ is clear since $f, g, h$ are piecewise linear inside each interval chosen during the construction of $I$. Recall that $f'(x)=g'(x)=0$  for every point $x\in (0,1)\setminus (I\cup A)$. Fix such an $x$. Choose corresponding $N(x)>0$  such that $n>N(x)$ implies
\[d(x,I_{1}\cup \cdots \cup I_{n})\geq \lambda_{n}.\]
As in the proof of Proposition \ref{lpf}, given any $t>0$ sufficiently small we may choose $n>N(x)$ such that
\[\lambda_{n+1}\leq t<\lambda_{n}\leq 1/2^n,\]
which implies
\[[x-t,x+t]\cap (I_{1}\cup \cdots \cup I_{n})=\varnothing.\]
Again it follows that $[x-t,x+t]$ can intersect at most $2^{k-n+1}$ intervals from $I_{k}$ for $k>n$. Recalling that $f'(x)=0$ at every point of $(0,1)\setminus (I\cup A)$, we have
\begin{align*}
\frac{\mathcal{L}^1(\{y\in [x-t,x+t]:f'(y)>0\})}{2t} &\leq \frac{\mathcal{L}^1([x-t,x+t] \cap I)}{2t}\\
&\leq \frac{1}{2\lambda_{n+1}} \sum_{k=n+1}^{\infty} 2^{k-n+2}w_{k}.
\end{align*}
Since the previous line converges to $0$ as $n\to \infty$, it follows $f'$ is approximately differentiable at $x$ with approximate derivative $0$. The argument for $g$ is the same. For $h$ we recall that $h'=2(f'g-g'f)$ almost everywhere. Combining this with the fact $f'(x)=g'(x)=0$ for every point $x\in (0,1)\setminus (I\cup A)$ gives $h'(x)=0$ for almost every $x\in (0,1)\setminus (I\cup A)$. For such $x$ the same argument as above applies, giving the desired conclusion. 
\end{proof}

\subsection{No $C^2$ Horizontal Lusin Approximation}

\begin{proposition}\label{noLusin}
The curve $\Gamma$ does not have the $C^2$ horizontal Lusin approximation property.
\end{proposition}

We will prove Proposition \ref{noLusin} by contradiction. Suppose $\Gamma$ does have the $C^2$ horizontal Lusin approximation property. Fix $\theta>4/5+1/31$ and a $C^2$ horizontal curve $\widetilde{\Gamma}=(F,G,H)\colon [0,1]\to \bbH$ such that the set
\[E:=\{t\in [0,1]: \widetilde{\Gamma}(t)=\Gamma(t)\}\]
satisfies $\mathcal{L}^1(E)>\theta$. Since $\mathcal{L}^1(I)<1/31$ by \eqref{mI}, we have $\mathcal{L}^1(E\setminus I)>4/5$.

\begin{lemma}\label{zeros}
Suppose $x\in E\setminus I$ is a Lebesgue density point of $E\setminus I$. Then
\[F(x)=F'(x)=F''(x)=0 \quad \mbox{and} \quad G(x)=G'(x)=G''(x)=0.\]
\end{lemma}

\begin{proof}
Let $x$ be as in the statement of the lemma. Then $F(x)=f(x)$ because $x\in E$ and $f(x)=0$ because $x\notin I$; hence $F(x)=0$. Since $x$ is a Lebesgue density point of $E\setminus I$ there exist $x_{n}\in E\setminus I$ with $x_{n}\to x$. By the same argument as before we have $F(x_{n})=0$ for every $n$. Hence $x_{n}\to x$ and $F(x_{n})=0=F(x)$ for every $n$. Since $F$ is $C^2$ this implies $F'(x)=F''(x)=0$. The argument for $G$ is the same.
\end{proof}

Since $\widetilde{\Gamma}$ is $C^2$, $F''$ and $G''$ are uniformly continuous on $[0,1]$. Fix $\delta>0$ such that
\begin{equation}\label{uc}
|F''(x)-F''(y)|<1 \mbox{ and }|G''(x)-G''(y)|<1 \qquad \mbox{for }|x-y|<\delta.
\end{equation}

\begin{lemma}\label{smallarea}
Suppose $a,b\in E\setminus I$ are Lebesgue density points of $E\setminus I$ and $|b-a|<\delta$. Then
\[|H(b)-H(a)|\leq 4|b-a|^{4}.\]
\end{lemma}

\begin{proof}
Since $F(a)=F(b)=0$ by Lemma \ref{zeros} and $\widetilde{\Gamma}$ is horizontal, we have
\[H(b)-H(a)=2\int_{a}^{b}(F'G-G'F)=4\int_{a}^{b}F'G.\]
We have $F(a)=F'(a)=F''(a)=0$ by Lemma \ref{zeros} and $|F''(t)-F''(a)|<1$ for $t\in [a,b]$ by \eqref{uc}. Hence $|F'(t)|\leq b-a$ and $|F(t)|\leq (b-a)^{2}$ for $t\in [a,b]$. The same estimates hold for $G$. This gives
\[|H(b)-H(a)| \leq 4(b-a)(b-a)(b-a)^{2}=4(b-a)^{4}.\]
\end{proof}


\begin{lemma}\label{goodpair}
For all $n\in \bbN$, there exists a pair $x, y\in (0,1)$ with the following properties:
\begin{itemize}
\item $x, y\in E\setminus I$ and are Lebesgue density points of $E\setminus I$,
\item $|x-y|\leq 1/2^{n}$,
\item $x, y$ are on opposite sides of an interval chosen in the construction of $I_{n+1}$.
\end{itemize}
\end{lemma}

\begin{proof}

We argue by contradiction. Assume there exists $n\in \bbN$ for which there is no pair $x$ and $y$ with the desired properties. Fix such an $n$. We consider intervals of the form $[L/2^{n}, (L+1)/2^{n}]$ for different integers $0\leq L<2^{n}$.

Suppose the interval $[L/2^{n}, (L+1)/2^{n}]$ has midpoint $(2L+1)/2^{n+1}$ which is the center of an interval $J$ chosen in the construction of $I_{n+1}$. The interval $J$ separates $[L/2^{n}, (L+1)/2^{n}]\setminus J$ into two subintervals $J_{1}$ and $J_{2}$ each of measure greater than $(1/3)(1/2^{n})$. Since there is no pair $x, y\in (0,1)$ with the properties in the statement of the lemma, in particular there is no such pair in the interval $[L/2^{n}, (L+1)/2^{n}]$. Hence either $J_{1}$ or $J_{2}$ does not contain any points of $E\setminus I$ which are Lebesgue density points of $E\setminus I$. Hence we have
\[\mathcal{L}^1((E\setminus I)\cap [L/2^{n}, (L+1)/2^{n}])\leq (2/3)\mathcal{L}^1([L/2^{n}, (L+1)/2^{n}]).\]

We now estimate the total measure of those intervals $[L/2^{n}, (L+1)/2^{n}]$ whose midpoint $(2L+1)/2^{n+1}$ is not chosen in the construction of $I_{n+1}$. Fix such an interval $[L/2^{n}, (L+1)/2^{n}]$. Then
\[B((2L+1)/2^{n+1},w_{n+1})\cap (I_{1}\cup \cdots \cup I_{n}) \neq \varnothing.\]
Different intervals of the form $B(k/2^{n+1},w_{n+1})$ are separated by a distance
\[1/2^{n+1}-2w_{n+1}\geq 1/2^{n+1}-2/2^{6n}\geq 1/2^{n+2}.\]
If an interval of length $T$ intersects $K$ intervals of the form $B(k/2^{n+1},w_{n+1})$ then we must have $T\geq K/2^{n+2}$, so $K\leq 2^{n+2}T$. The set $I_{i}$ is a union of $2^{i-1}$ intervals of length $2w_{i}$. Hence the number of intervals $B(k/2^{n+1},w_{n+1})$ which intersect $I_{1}\cup \cdots \cup I_{n}$ can be estimated by
\[\sum_{i=1}^{n}2^{i-1}2^{n+2}2w_{i}=2^{n+2}\sum_{i=1}^{n}2^{i}w_{i}.\]
Hence the total measure of all those intervals $[L/2^{n}, (L+1)/2^{n}]$ whose midpoint is not chosen in the construction of $I_{n+1}$ can be estimated by 
\[(1/2^n)2^{n+2}\sum_{i=1}^{n}2^{i}w_{i}=4\sum_{i=1}^{n}2^{i}w_{i} \leq 4\sum_{i=1}^{\infty}2^{i}/2^{6i}=4/31.\]

Let $G$ be the collection of those integers $L$ such that the midpoint of $[L/2^{n}, (L+1)/2^{n}]$ is the center of an interval chosen in the construction of $I_{n+1}$. Let $B$ be the collection of those integers $L$ such that the midpoint of $[L/2^{n}, (L+1)/2^{n}]$ is not chosen.
We estimate as follows
\begin{align*}
\mathcal{L}^1(E\setminus I) &= \sum_{L\in G} \mathcal{L}^1((E\setminus I)\cap [L/2^{n}, (L+1)/2^{n}])\\
&\qquad \qquad + \sum_{L\in B} \mathcal{L}^1((E\setminus I)\cap [L/2^{n}, (L+1)/2^{n}])\\
&\leq \sum_{L\in G} (2/3)\mathcal{L}^1([L/2^{n}, (L+1)/2^{n}]) + \sum_{L\in B} \mathcal{L}^1([L/2^{n}, (L+1)/2^{n}])\\
&\leq 2/3 + 4/31\\
&\leq 4/5.
\end{align*}
Since $\mathcal{L}^1(E\setminus I)>4/5$ we obtain a contradiction which proves the lemma.
\end{proof}

We now derive a contradiction which proves Proposition \ref{noLusin}. Recall $\delta>0$ from \eqref{uc} and the fact that $4^{n}h_{n}\to \infty$. Using Lemma \ref{goodpair}, we may fix $n$ with $1/2^{n}<\delta$ and $4^{n}h_{n+1}\geq 2$ for which there exist points $x, y\in (0,1)$ with $x<y$ such that
\begin{itemize}
\item $x, y\in E\setminus I$ and are Lebesgue density points of $E\setminus I$,
\item $|x-y|\leq 1/2^{n}$,
\item $x, y$ are on opposite sides of an interval chosen in the construction of $I_{n+1}$.
\end{itemize}
Since $|x-y|\leq 1/2^{n}<\delta$ and $x,y\in E\setminus I$ are Lebesgue density points of $E\setminus I$, we have by Lemma \ref{smallarea}
\begin{equation}\label{smallH}
|H(y)-H(x)|\leq 4|y-x|^{4}\leq 4/16^{n}.
\end{equation}
Since $x< y$ are on opposite sides of an interval chosen in the construction of $I_{n+1}$, we have by Lemma \ref{hinc}
\begin{equation}\label{bigh}
h(y)-h(x)\geq 4h_{n+1}^{2}.
\end{equation}
Since $x,y\in E$ we have $H(y)-H(x)=h(y)-h(x)$. Combining this with \eqref{smallH} and \eqref{bigh} gives $h_{n+1}^{2} \leq 1/16^{n}$ or equivalently $4^{n}h_{n+1}\leq 1$. This contradicts the choice of $n$ with $4^{n}h_{n+1}\geq 2$, proving Proposition \ref{noLusin} and hence proving Theorem \ref{thm2}.

\end{document}